\documentclass[a4paper,11pt]{article}
\usepackage{amsfonts}
\usepackage{amsmath}
\usepackage{amsthm}
\usepackage{amsfonts}
\usepackage{latexsym}
\usepackage{amssymb}
\usepackage{dsfont}

\newtheorem{teo}{Theorem}
\newtheorem{lem}[teo]{Lemma}
\newtheorem{cor}[teo]{Corollary}
\newtheorem{pro}[teo]{Proposition}
\newtheorem{rem}[teo]{Remark}
\newtheorem{exa}[teo]{Example}

\newcommand{\F}{\mathcal{F}}
\newcommand{\cE}{\mathcal{E}}
\newcommand{\cG}{\mathcal{G}}
\newcommand{\cH}{\mathcal{H}}
\newcommand{\cK}{\mathcal{K}}
\newcommand{\cU}{\mathcal{U}}
\newcommand{\cP}{\mathcal P}
\newcommand{\cF}{\mathcal F}
\newcommand{\cS}{\mathcal S}
\newcommand{\cT}{\mathcal T}
\newcommand{\cM}{\mathcal M}
\newcommand{\cN}{\mathcal N}
\newcommand{\cX}{\mathcal X}
\newcommand{\cY}{\mathcal Y}

\newcommand{\NN}{\mathbb{N}}
\newcommand{\CC}{\mathbb{C}}
\newcommand{\RR}{\mathbb{R}}

\newcommand{\la}{\lambda}
\newcommand{\eps}{\varepsilon}

\newcommand{\inc}{\subseteq}
\newcommand{\rai}{^{1/2}}
\newcommand{\mrai}{^{-1/2}}
\newcommand{\U}{\cU (\cH )}
\newcommand{\pint}[1]{\displaystyle \left \langle #1 \right\rangle}
\newcommand{\nui}[1]{|\!|\!| #1 |\!|\!|}
\newcommand{\QED}{\hfill $\square$}

\DeclareMathOperator{\tr}{tr}
\DeclareMathOperator{\ind}{ind}
\DeclareMathOperator{\rk}{rk}
\begin{document}

\title{Procrustes problems and Parseval quasi-dual frames}

\author{{G. Corach$^{a,b}$ \ , \ P. Massey$^{b,c}$  \ , \ M. Ruiz$^{b,c}$} \\ \\  {\small  $^a$ Universidad de Buenos Aires, Argentina} \\ {\small} \\ {\small $^b$ Instituto Argentino de Matem\'atica ``A.P. Calder\'on'', Argentina}\\ {\small} \\ {\small $^c$ Universidad de La Plata, Argentina}}
\date{}

\maketitle

\begin{abstract} 
Parseval frames have particularly useful properties, and in some cases, they can be used to reconstruct signals which were analyzed by a non-Parseval frame. In this paper, we completely describe the degree to which such reconstruction is feasible. 
Indeed, notice that for fixed frames $\cF$ and $\cX$ with synthesis operators $F$ and $X$, the operator norm of $FX^*-I$ measures the (normalized) worst-case error in the reconstruction of vectors when analyzed with $\cX$ and synthesized with $\cF$. Hence, for any given frame $\cF$, we compute explicitly the infimum of the operator norms of $FX^*-I$, where $\cX$ is any Parseval frame. The $\cX$'s that minimize this quantity are called Parseval quasi-dual frames of $\cF$. Our treatment considers both finite and infinite Parseval quasi-dual frames.
\end{abstract}

{\bf keywords:}{ Procrustes problems , Parseval frames , Dual frames.}

{AMS subclass:} 42C15 ; 15A60 ; 65F35. 

\section{Introduction}
Let $\cH$ be a separable complex Hilbert space. A sequence $\cF=\{f_i\}_{i\in \NN}$ is a frame for $\cH$ if there exist constants $A,B>0$ such that
\begin{equation}\label{def frame}
A\|f\|^2\leq \sum_{i\in \NN} \, |\pint{f,f_i}|^2\leq B\|f\|^2, \quad \forall f\in \cH.
\end{equation}
The biggest $A$ and the least $B$ with the properties above are called the {\it optimal frame bounds} of the frame $\cF$ and denoted $A_\F$ and $B_\F$ respectively. If $A_\F=B_\F=1$ then $\cF$ is called {\sl Parseval frame}.
Roughly speaking, a frame $\cF$ is a generator set for $\cH$ when we allow linear combinations of elements of $\cF$ with coefficients in $\ell_2(\NN)$.

More precisely, given a frame $\cF$ define the synthesis operator of $\cF$ by
$F:\ell_\NN^2\rightarrow \cH,$  $F((\alpha_i)_{i\in\NN})=\sum_{i\in\NN} \alpha_i f_i$, and the analysis operator of $\cF$ by $F^*:\cH \rightarrow \ell_\NN^2$, i.e. 
$F^*f=(\pint{f,f_i})_{i\in \NN}$. Notice that the inequality to the right in Eq. \eqref{def frame} implies that $F$ (and hence $F^*$) is a well defined, bounded linear transformation.
Moreover, if we let $S_\cF=FF^*\in B(\cH)^+$ then 
\[
S_\cF f=FF^* f = \sum_{i\in \NN}\, \pint{f,f_i}f_i, \quad \forall f\in \cH
\] and the defining inequalities \eqref{def frame} are equivalent to the operator inequalities
$A\cdot I\leq S_\cF\leq B\cdot I.$ 
Therefore $S_\cF=FF^*$, the so-called frame operator of $\cF$, is a positive invertible operator. 
 On the other hand, the identities
\begin{equation}\label{ident recons}
f=S_\cF S^{-1}_\cF f=\sum_{i\in \NN}\, \pint{S_\cF^{-1}f,f_i} f_i =\sum_{i\in \NN}\, \pint{f,S_\cF^{-1}f_i} f_i \, , \quad \forall f\in \cH 
\end{equation} show that $\cF$ allows for linear representations of vectors in $\cH$, where the sequence of coefficients $(\langle f, S^{-1}_\cF f_i\rangle)_{i\in\NN}=F^*(S^{-1}_\cF\,f)\in\ell_2(\NN)$. Also notice that, if we let $\cF^\#=\{S^{-1}_\cF\,f_i\}_{i\in\NN}$ then $\cF^\#$ is a frame for $\cH$ with frame operator $S_\cF^{-1}$. 
This shows that coefficients given by the map $\cH\ni f\mapsto (\langle f, S^{-1}_\cF f_i\rangle)_{i\in\NN}\in\ell_2(\NN)$ depend continuously on $f\in\cH$ i.e. $\cF$ allows for linear and stable reconstruction 
formulas.

An important feature of frames $\cF$ for $\cH$ is their redundancy. Indeed, it is well known (see \cite{Crisbook,Grochbook}) that the lack of uniqueness in the linear representation induced by a frame is an advantage in this context, as it provides tools to deal with problems that arise in applications of frame theory.  A typical measure of redundancy of a frame $\cF$ is given by the excess of $\cF$, which is the dimension of the nullspace of its synthesis operator $F$. In case $\cF$ has nonzero excess then there exist infinitely many dual frames for $\cF$ i.e. frames $\cG=\{g_i\}_{i\in\NN}$ such that 
\[f=\sum_{i\in\NN} \pint{f,g_i}\, f_i\  , \quad \forall f\in \cH\  . \]
Notice that the frame $\cF^\#$ defined above is a dual frame for $\cF$, called the canonical dual frame.

It is easy to see that the set of dual frames for $\cF$ is in bijection with the set of bounded linear operators $G:\ell_\NN^2\rightarrow \cH$ such that $FG^*=I_\cH$, via the map $G\mapsto \cG=\{Ge_i\}$. This last fact incidentally shows that if $\cG$ is a dual frame for $\cF$ then $\cF$ is also a dual frame for $\cG$  i.e. 
\[f=\sum_{i\in\NN} \pint{f,g_i}\, f_i=\sum_{i\in\NN} \pint{f,f_i}\, g_i \  , \quad \forall f\in \cH\  . \]
 In this case we say that $(\cF,\cG)$ is a dual pair of frames for $\cH$.
  It is worth pointing out that the canonical dual frame $\cF^\#$ 
corresponds to the adjoint of the Moore-Penrose pseudo-inverse of $F$ under the previous bijection; this fact indicates that $\F^\#$ plays a key role among the set of all dual frames for $\cF$. 

Nevertheless, given a redundant frame $\cF$ then the canonical dual frame $\cF^\#$ is not always the best choice for a dual of $\cF$. 
For example, numerical stability comes into play when dealing with reconstruction formulas derived from a dual pair $(\cF,\cG)$; in this case a measure of stability of the reconstruction algorithm for fixed $\cF$ is given by the condition number of the frame operator $S_\cG$ of $\cG$. It is known \cite{MRS} that the dual frames $\cG$ that minimize this condition number are different (in general) from the canonical dual (see \cite{BLem,LH,WES} for other examples of this phenomena). 
In this vein, D. Han \cite{Han} characterized those frames $\cF$ for which there exists a dual frame for $\cF$, denoted by $\cX=\{x_i\}_{i\in\NN}$, which is a Parseval frame i.e. for which the frame bounds coincide with 1. Thus, in this case $S_\cX=I$ and therefore the condition number of $S_\cX$ is minimum. The conditions found in \cite{Han} for the existence of a Parseval dual of a frame $\cF=\{f_i\}_{i\in\NN}$ turn out to be equivalent to the conditions
for the existence of a larger Hilbert space $\cK\supset \cH$, an orthonormal basis $\{k_i\}_{i\in\NN}$ of $\cK$ and an oblique projection $Q:\cK\rightarrow \cH$ such that $Qk_i=f_i$, for $i\in \NN$,  found by
J. Antezana, G. Corach, M. Ruiz and D. Stojanoff in \cite{ACRS}.

In case a frame $\cF$ does not admit a Parseval dual frame, then we can consider one of the following alternatives: we can search for dual frames that are optimal for numerical stability (or such that they are less complex to compute) i.e. obtaining (theoretical) perfect reconstruction formulas for less stable dual frames - or we can search for Parseval frames $\cX$, which are optimally stable, that minimize the reconstruction error when considering the reconstruction formulas derived from the pair $(\cF,\cX)$. 
In this paper we shall consider the second approach and search for Parseval frames $\cX$ that minimize the (normalized) worst case reconstruction error for a (blind) reconstruction algorithm derived from the pair $(\cF,\cX)$, (formally) following the recent analytic scheme from the theory of optimal dual frames for erasures of a fixed frame \cite{LHan,LHantin,LH,MRSrob}. Explicitly,
we search for Parseval frames  $\cX=\{x_i\}_{i\in\NN}$ - i.e. such that $XX^*=I_\cH$ - such that they minimize the worst case reconstruction error
\[ \sup_{f\in\cH\, , \ \|f\|=1} \| \sum_{i\in\NN} \pint{f,x_i}\, f_i - f\| =\|FX^*-I_\cH\|\ \]
where $\|FX^*-I_\cH\|$ stands for the operator norm on $B(\cH)$. To this end we introduce the optimal bound 
\[ \alpha(F)=\inf \{  \|FX^*-I\|\, , \ {XX^*=I_\cH}\} \ .\] 
As we shall see, 
$\alpha(F)$ depends on the spectrum of the frame operator $S_\cF$ as well as on the excess of $\cF$.
In case the optimal bound $\alpha(F)$ is attained (e.g. finite frames for finite dimensional Hilbert spaces), 
we introduce the set of Parseval quasi-dual frames 
defined as 
\[\mathfrak{X}(F)=\{ \cX=\{x_i\}_{i\in\NN}: \; XX^*=I, \, \alpha(F)=\|FX^*-I\|\}.\]
In these cases, $\alpha(F)$ is the worst case reconstruction error 
of an encoding-decoding algorithm based on an optimal pair $(\cF,\cX)$, where $\cX$ is Parseval quasi-dual frames for $\cF$.

If  $\cH$ is finite dimensional, the previous problem fits into the setting of the so called Procrustes problems, which are relevant in statistics, shape theory, numerical analysis and optimization, among other disciplines. The reader is referred to the book by Gower and Dijksterhuis \cite{GoDi} and to the survey \cite{NHi} of Higham, for many references (and to learn about the myth of the bandit Procrustes). We should also mention the papers by Eld\'en and Parks \cite{ElPa}, the results of Watson \cite{Wat1,Wat2} and Mathias \cite{Mat}, who showed that $\alpha$ effectively depends on the norm. Kintzel \cite{Kni} dealt with Procrustes problems in finite dimensional Krein spaces, and Peng, Hu, and Zhang \cite{PeHuZh} with weighted Procrustes problems. All these references study the problems in finite dimensional spaces. The present paper seems to be the first to relate problems in frame theory in arbitrary Hilbert spaces with Procrustes techniques.

The paper is organized as follows. In Section \ref{sec prelims} we revise some of the main results from \cite{ACRS} and \cite{Han} that characterize the existence of Parseval dual frames of a frame $\F$ for $\cH$.
In Section \ref{sec finite frames} we consider the finite frames for a finite dimensional Hilbert space $\CC^n$; hence, we study the problem of computing the optimal bound $\alpha_{\nui{\cdot}}(\F)$ for an arbitrary unitarily invariant norm (u.i.n.) $\nui{\cdot}$ for $M_n(\CC)$. In particular, we show that there exists $\cX$ that is a Parseval $\nui{\cdot}$-quasi dual, for every u.i.n. In Section \ref{sec 4} we consider the computation of $\alpha(\F)$ for frames $\F$ in a separable infinite dimensional (complex) Hilbert space $\cH$.
In case the excess of $\F$ is infinite, we compute $\alpha(\F)$ explicitly and show that $\F$ always admits Parseval quasi-duals $\cX$ which also satisfy that $FX^*$ is a multiple of the identity. This last fact allows us to derive simple reconstruction formulas from the pair $(\F,\cX)$. 
 If $\F$ has finite excess, then we obtain an explicit formula for $\alpha(\F)$ that depends on the spectrum of the Gramian operator $F^*F$ and the excess of $\F$. For the convenience of the reader, we present several technical results related with the computation of $\alpha(\F)$ in this case in a separate Appendix (Section \ref{App}): our strategy relies on Rogers' work \cite{Rogers} on approximation of bounded operators by unitary operators.
 
\section{Preliminaries}\label{sec prelims}

{\bf Notations and terminology}. In what follows, $\cH$ denotes a separable complex Hilbert space. If $\cK$ is another Hilbert space then $B(\cK,\cH)$ denotes the Banach space of bounded linear transformations from $\cK$ to $\cH$. If $\cK=\cH$ then we denote $B(\cH)=B(\cH,\cH)$ the algebra of bounded linear operators acting on $\cH$, endowed with the operator norm. We denote by $M_{n,m}(\CC)$ the space of $n\times m$ complex matrices and identify
$M_{n,m}(\CC)=B(\CC^m,\CC^n)$. If $m=n$ then we write $M_n(\CC)=M_{n,n}(\CC)$. 
If $F\in B(\cK,\cH)$ then we denote by $N(F)$ and $R(F)$ the nullspace and range of $F$ respectively.
We denote by $\U$ the group of unitary operators acting on $\cH$.
  
  \smallskip

\noindent {\bf Frames in Hilbert spaces}. As mentioned in the Introduction, a sequence $\{f_i\}_{i\in \NN}$ in a Hilbert space $\cH$ is a {\it frame} for  $\cH$ if there exist $A,B>0$ such that Eq. \eqref{def frame} holds.
The optimal constants $A_\cF,\,B_\cF$ are called the  {\it optimal frame bounds}. If the frame bounds of $\cF$ are equal to 1 we say that $\cF$ is {\it  Parseval}. More generally, we say that $\cF$ is tight if the frame bounds coincide. 

       Given a frame $\cF=\{f_i\}_{i\in \NN}$ for $\cH$, let $F\in B(\ell_2, \cH)$,  $F^*\in B(\cH, \ell_2)$ and $S_\cF=FF^*\in B(\cH)^+$ denote the synthesis, analysis and frame operators of $\cF$ respectively.  Notice that $A,B>0$ satisfy Eq. \eqref{def frame} if and only if $A\,I\leq S_\cF\leq B\,I$. Hence, the optimal frame bounds $0<A_\cF\leq B_\cF$ are given by                   
\begin{equation}\label{cotas optimas} A_\cF \rai =\gamma(F):=\inf\{\|Fx\|\, : \; x\in N(F)^\perp, \, \|x\|=1\}=\|F^\dagger\|^{-1}  \quad \text{ and } \quad  B_\cF\rai=\|F\|,
\end{equation}
where $F^\dagger$ denotes the Moore-Penrose pseudoinverse of $F$ and $\gamma(F)$  its {\sl reduced minimum modulus}. 

These facts show that a frame $\cX=\{x_i\}_{i\in\NN}$ is Parseval if and only if $S_\cX=I_\cH$ or equivalently, if the synthesis operator $X\in B(\ell_2,\cH)$ is a coisometry. It is well known that Parseval frames have several nice properties; for instance, the spread of the eigenvalues of the frame operator $S_\cX$ is minimum, a fact that implies numerical stability of linear encoding-decoding schemes derived from $\cX$. 
On the other hand, the canonical dual of a Parseval frame $\cX$ coincides with $\cX$ so that the identities in Eq. \eqref{ident recons} become 
\[ f=\sum_{i\in\NN} \pint{f,x_i}\,x_i \ , \quad f\in\cH\, , \]
which is formally analogous to the linear representation derived from an orthonormal basis for $\cH$. 
Indeed, there is a close connection between Parseval frames and orthonormal bases, since the elements of a Parseval frame can be seen as the image of a orthonormal basis of a larger Hilbert space under a orthogonal projection (\cite{HanLar}).
This kind of  dilation property of Parseval frames was extended in \cite{ACRS} where it is shown that $A_\cF\geq 1$ and $\dim \overline{ R(S_\cF-I)}\leq \dim N(F)$ are necessary and sufficient conditions for the frame $\cF$ in order to be the image of a orthonormal basis under a oblique projection (see also \cite[10.4.]{HKLW} which is largely based in \cite{ACRS} and contains an elementary approach to these results). 
Years later, D. Han proved that those conditions also ensure the existence of a Parseval dual frame $\cX$ for $\cF$ (\cite{Han}).
We summarize all these results in the following:
\begin{teo}\label{hay dual}
Let $\cF=\{f_i\}_{i\in \NN}$ be a frame for $\cH$, with optimal frame bounds $0<A_\cF\leq B_\cF$ and synthesis operator $F:\ell_2\rightarrow \cH$. Then the following statements are equivalent:
\begin{enumerate}
\item There exists a Parseval frame $\cX=\{x_i\}_{i\in \NN}$ such that $x=\sum_{i\in \NN} \pint{x,x_i}\, f_i$, $\forall x\in \cH$.
\item \label{condiciones} $A_\cF\geq 1$ and $\dim \overline{ R(S_\cF-I)}\leq \dim N(F)$.
\item There exists a Hilbert space $\cK\supset \cH$ with an orthonormal basis $\{k_i\}$ and a oblique projection $Q\in B(\cK)$ onto $\cH$ such that $f_i=Qk_i$, for every $i\in \NN$.
\end{enumerate}
\QED
\end{teo}

\section{Parseval quasi-dual frames: finite frames}\label{sec finite frames}

Let $\F=\{f_i\}_{i=1}^m$ be a finite frame for $\cH=\CC^n$. As we have seen in the previous section, in some cases $\F$ admits
a dual Parseval frame $\cX=\{x_i\}_{i=1}^m$. Indeed, in the finite dimensional context
 the equivalence of items 1. and 2. in Theorem \ref{hay dual} imply the following:
\begin{pro}\label{prop 4.1.}Let $\F=\{f_i\}_{i=1}^m$ be a frame for $\cH=\CC^n$ with synthesis operator $F:\CC^m\rightarrow \CC^n$.
Let $\la=(\lambda_i)_{i=1}^n \in \RR^n$ be the eigenvalues of the frame operator $S_\F$, arranged in non-increasing order and counting multiplicities. 
Then $\F$ has a Parseval dual if and only if
\begin{enumerate}
\item $\la_n\geq 1$;
\item \label{item2} $2n-m\leq \dim N(S_\F-I)$.
\end{enumerate}
In particular, if $m\geq 2n$, the existence of a Parseval dual is guaranteed by $\la_n\geq 1$.
\end{pro}
\begin{proof}
Notice that $A_\cF=\lambda_n$ and that  $\dim N(F)=m-n$ since $F$ is surjective. Then the proof  follows by dimension arguments.
\end{proof}
In case $m<2n$, the condition \ref{item2} in Proposition \ref{prop 4.1.} shows that the existence of a Parseval dual for $F$ is related to the multiplicity of 1 as an eigenvalue of $S_\F$.

\medskip

In case $\F$ has no Parseval dual frames, it is natural to consider those Parseval frames $\cX$ which minimize the reconstruction error based on an encoding-decoding scheme in terms of $\cX$ and $\F$. Following ideas from \cite{NHi},  the measure of the reconstruction error that we shall consider in this section is $\nui{FX^*-I}$, where $\nui{\cdot}$ is a unitarily invariant norm (u.i.n.) in $M_n(\CC)$ (see \cite{Bhatia} for a detailed account on these norms) and $F$ and $X$ denote the synthesis operators of $\F$ and $\cX$ respectively. That is, we are interested in computing the optimal bound 
\begin{equation}\label{defi alpha nui}
 \alpha_{\nui{\cdot}}(\F):=\min\{\nui{FX^*-I}:\ XX^*=I \  \}
\end{equation} 
and determining the optimal  Parseval quasi-dual frames which are the Parseval frames that attain this lower bound i.e.,  
\begin{equation}\label{defi alpha quasi duals}
\mathfrak{X}_{\nui{\cdot}}(\F):=\{\cX=\{x_i\}_{i=1}^m: \ XX^*= I \ , \ \alpha_{\nui{\cdot}}(\F)=\nui{FX^*-I}\}\ .
\end{equation} 
In what follows, we give a detailed description of the optimal bound $\alpha_{\nui{\cdot}}(\F)$ and the set of optimal 
 Parseval quasi-dual frames $\mathfrak{X}_{\nui{\cdot}}(\F)$ for a frame $\F$ and an arbitrary u.i.n.

\medskip

Let us begin by recalling the following Procrustes-type result due to Fan and Hoffman \cite{FanHof}.
\begin{teo}[\cite{FanHof}]\label{teo FanHof}
Let $A,\,U\in M_n(\CC)$ with $U$ unitary and such that $A=U\,|A|$, where $|A|=(A^*A)^{1/2}$.
Then, for any unitary matrix $W\in M_n(\CC)$ and any u.i.n. $\nui{\cdot}$ in $M_n(\CC)$ we have that 
\[\nui{A-U}\leq \nui{A-W}\leq \nui{A+U}\  .\] 
\end{teo}

\begin{lem}\label{lem nui}
Let $\F=\{f_i\}_{i=1}^m$ be a frame for $\CC^n$ with synthesis operator $F$. Let $\cS$ be an $n$-dimensional subspace of $\CC^m$ and let $C_\cS=\{Y\in M_{n,m}(\CC) \, : \ Y^*Y=P_\cS\}$. Let $\nui{\cdot}$ be a u.i.n. on $M_n(\CC)$ and let $X\in C_\cS$ be such that 
\[\nui{\, FX^*-I \,}=\min_{Y\in C_\cS}\nui{\, FY^*-I\,}\ .\]
 In this case we have that 
\[ \nui{\, FX^*-I \,} =\nui{ \ |FX^*| - I \ }  \ .\]
\end{lem}
\begin{proof} It is well known that in finite dimensions we can  factorize $FX^*=U\,|FX^*|$, where $U\in M_n(\CC)$ is a unitary matrix. Then
\[
\nui{ FX^*-U}=\nui{ (FX^*U^*-I)\,U}=\nui{FX^*U^*-I}\geq \nui{FX^*-I}\geq \nui{FX^*-U}
\]
where the first inequality follows from the fact that $Y=UX\in C_\cS$ is such that $Y^*=X^*U^*$ (and the minimality of $X$) and the second inequality follows from Theorem \ref{teo FanHof} (and the fact that $I$ is a unitary matrix). Hence \[\nui{FX^*-I}=\nui{FX^*U^*-I}=\nui{ U\,( |FX^*| - I)\,U^* }=\nui{ \  |FX^*| - I \, }\ . \]
\end{proof}

\begin{rem}\label{rem uso del lem}
Let $\F=\{f_i\}_{i=1}^m$ and $\cX=\{x_i\}_{i=1}^m$ be frames for $\CC^n$ with synthesis operators $F$ and $X$, and assume that $\cX$ is an optimal  Parseval quasi-dual frame for some u.i.n. $\nui{\cdot}$ i.e. such that $\alpha_{\nui{\cdot}}(\F)=\nui{FX^*-I}$. If we let $\cS=R(X^*)\inc\CC^m$, $\dim \cS=n$, then Lemma \ref{lem nui} 
implies that 
\[ \nui{\, FX^*-I \, }=\nui{ \ |FX^*| - I \ } \ .\]
 Notice that the singular values of the Hermitian matrix $|FX^*| - I\in M_n(\CC) $ are determined by the eigenvalues of $|FX^*|$. Moreover, $|FX^*|^2=X(F^*F)X^*$ and therefore, the eigenvalues of $|FX^*|^2$ coincide with the first $n$ entries of the vector $\lambda(P_\cS(F^*F)\,P_\cS)$ of eigenvalues of the compression of the Gramian matrix $F^*F$ to the subspace $\cS$. This last fact shows a connection between our problem and the following theorem of Fan-Pall related with eigenvalues of compressions.
\end{rem}

\begin{teo}[\cite{FanPall}]\label{FanPallTeo}
Let $m\geq n$ and $\mu_1\geq \ldots\geq\mu_n$, $\la_1\geq \ldots \geq \la_m$  be $n+m$ real numbers. There exists an Hermitian matrix $H\in M_m(\CC)$ with eigenvalues $(\la_i)_{i=1}^m$
 and a coisometry $U\in M_{n,m}(\CC)$ such that $UHU^*$ has eigenvalues $(\mu_i)_{i=1}^n$ if and only if
\begin{equation}\label{FanPall Ineq} \la_i \geq \mu_i, \qquad \la_{m-n+i}\leq \mu_i, \qquad 1\leq i\leq n.
\end{equation}
\end{teo}
In particular, given a positive semidefinite matrix $A\in M_m(\CC)$ with eigenvalues $\la_1\geq\ldots \geq \la_m\geq 0$, there exists a rank-$n$ orthogonal projection $P\in M_m(\CC)$ such that the first $n$ eigenvalues of $PAP$ are $\mu_1\geq\ldots\geq \mu_n$ if and only if the inequalities in \eqref{FanPall Ineq} - the so-called Fan-Pall inequalities - hold.

\begin{teo}\label{dim finita2}
Let $\F=\{f_i\}_{i=1}^m$ be a frame for $\CC^n$ with synthesis operator $F$. 
Denote by  $\la=(\la_i)_{i=1}^m$  the eigenvalues of the Gramian $F^*F\in M_m(\CC)$ counting multiplicities and arranged in a non-increasing order. Consider the positive numbers $(d_{j})_{j=1}^n$ defined for $1\leq j\leq n$ 
\begin{equation}\label{optimal singular values}
d_{j} =
\left\{
	\begin{array}{ll}
		\la_{m-n+j}  & \mbox{if } \la_{m-n+j} \geq 1 \ ;\\
		1 & \mbox{if } \la_{m-n+j} < 1 \leq \la_j \ ; \\
		\la_j & \mbox{if } \la_j<1 \ . 
	\end{array}
\right.
\end{equation}
In this case we have that: 
\begin{enumerate}
\item There exists a Parseval frame $\cX=\{x_i\}_{i=1}^m$ for $\CC^n$  with synthesis operator $X$ such that $FX^*$ is positive semidefinite with eigenvalues $(d_j^{1/2})_{j=1}^n$.
\item For every u.i.n. $\nui{\cdot }$ and every Parseval frame $\cY=\{y_i\}_{i=1}^m$ with synthesis operator $Y$
we have that \begin{equation}\label{desi nui}
\nui{ \,FX^*-I\,} \leq \nui{\,FY^*-I\,} \ .
\end{equation}
\item If we let $\Phi:\RR^n\rightarrow \RR_{\geq 0}$ denote the symmetric gauge function associated with $\nui{\cdot}$ then \begin{equation}\label{alfa con nui}
 \alpha_{\nui{\cdot}}(\F)=\Phi((1-d_j^{1/2})_{j=1}^n)\ .
 \end{equation}
\end{enumerate}
\end{teo}
\begin{proof}
By Theorem \ref{FanPallTeo} it is clear that there exists an $n$-dimensional subspace $\cS\subseteq \CC^m$ such that the first $n$ eigenvalues of the compression $P_\cS (F^*F)P_\cS\in M_m(\CC)$ are given by the sequence of positive numbers $(d_j)_{j=1}^n$ arranged in non-increasing order. Notice that in this case $F P_ \cS F^*=|P_\cS F^*|^2\in M_n(\CC)$ is an invertible operator with eigenvalues $(d_j)_{j=1}^n$ (arranged in non-increasing order).

Consider the polar decomposition $P_\cS F^*=V \,|P_\cS \,F^*|$. Then, it is well known that $V^*P_\cS F^*=|P_\cS F^*|$ and hence $FX^*=|P_\cS F^*|$, where $X=V^*P_\cS\in M_{n,m}(\CC)$ is such that $XX^*=I$. Hence, if we let $\{e_i\}_{i=1}^m$ denote the canonical basis for $\CC^m$ and set $\cX=\{Xe_i\}_{i=1}^m$, then $\cX$ has the desired properties.

Let $\cY=\{y_i\}_{i=1}^m$ be a Parseval frame with synthesis operator $Y$. Notice that Lemma by \ref{lem nui} (see also Remark \ref{rem uso del lem}) we can assume that 
\begin{equation}\label{ecua teo 3.6} \nui{\, FY^*-I \, }=\nui{ \ |FY^*| - I \ } \ , \end{equation} otherwise, we replace $\cY$ with a Parseval frame $\tilde \cY$ with $ R(\tilde Y^*)=R(Y^*)$ and such that 
\[\nui{FY^*-I}\geq \nui{F(\tilde Y)^*-I} = \nui{\ |F(\tilde Y)^*| -I\ }  \ .\]
Since $\cY$ is a Parseval frame then the eigenvalues $(c_j)_{j=1}^n$ of $|F Y^*|\in M_n(\CC)$ correspond to the square roots of the first $n$ eigenvalues of the compression $P_{R(Y^*)} (F^*F) P_{R(Y^*)}\in M_m(\CC)$, for which the Fan-Pall inequalities in Theorem \ref{FanPallTeo} apply.
 In this case it is straightforward to check that 
 \[ |1-d_j^{1/2}|\leq |1-c_j^{1/2}| \ , \ \ 1\leq j\leq n \ .\]
The singular values  $s(FX^*-I)\in \RR^n$ (respectively $s(|FY^*|-I)\in\RR^n$) are $(|1-d_j^{1/2}|)_{j=1}^n$ arranged in non-increasing order (respectively, $(|1-c_j^{1/2}|)_{j=1}^n$ arranged in non-increasing order). Therefore, since  $\Phi$ is monotone - because it is gauge invariant (see \cite{Bhatia})- we obtain 
 equations \eqref{desi nui} and \eqref{alfa con nui}.
\end{proof}

\begin{cor}\label{cor dim fin}
With the notations of Theorem \ref{dim finita2} 
\[
\mathfrak{X}_{\nui{\cdot}}(\F):=\{\cX=\{x_i\}_{i=1}^m: \, \Phi((s_j(FX^*-I))_{j=1}^n)= \Phi((s_j(FX^*)-1)_{j=1}^n)=
\Phi((1-d_j^{1/2})_{j=1}^n)\}
\ .
\]
 Notice further that the singular values $(s_j(FX^*))_{j=1}^n$ correspond to the square roots of the first $n$ entries of the vector $\lambda(P_\cS(F^*F)P_\cS)\in\RR^m$, where $\cS=R(X^*)$.
\end{cor}

\begin{exa}
Let $1\leq p\leq \infty$ and consider the norm $\|\cdot\|_p$ on $M_n(\CC)$ i.e., if $T\in M_n(\CC)$ then 
\[\|T\|_p=\left (\sum_{i=1}^n s_i(T)^p\right)^{1/p} \ 1\leq p<\infty \quad \text{ and } \quad \|T\|_\infty=\|T\|\, . \] 
Then $\|\cdot\|_p$ is a unitarily invariant norm on $M_n(\CC)$. Let $\F=\{f_i\}_{i=1}^m$ be a frame for $\CC^n$ and let $(d_j)_{j=1}^n$ be the sequence defined in Theorem \ref{dim finita2}. If we denote by $\alpha_p(\F)= \alpha_{\|\cdot\|_p}(\F)$ then by Corollary \ref{cor dim fin} we get that
\[
\alpha_p(\F) =
\left\{
	\begin{array}{ll}

\left (\sum_{i=1}^n |1-d_j^{1/2}|^p\right)^{1/p}		 & \mbox{ if }  1 \leq p <\infty \ ; \\ \\ 
 \max\left\{1-\la_n\rai, \la_{m-n+1}\rai-1,0\right\} & \mbox{ if } p =\infty  \ . 
	\end{array}
\right.
\]
\end{exa}

\begin{rem}
Define  $1\leq r\leq n$ as the biggest integer such that $\la_r\geq 1$, if $\la_1<1$ we set $r=0$. A closer look at the definition of the optimal set of singular values $(d_j)_{j=1}^n$ given in \eqref{optimal singular values} shows that
%
%
  the nonzero values of $(|1-d_j\rai|)_{j=1}^n$ are the nonzero values of  
\begin{equation}\label{los dj revisados}
\left\{
	\begin{array}{ll}
		(1-\la_j\rai)_{j\geq r+1} & \mbox{if } r\leq m-n+1\ \mbox{ and } r<n;\\
	(|\la_j\rai -1|)_{j\geq m-n+1}& \mbox{if } r>m-n+1,\,  \\
	\end{array}
\right.
\end{equation}
If $r=n$ and $m-n+1>n$, we have that $d_j=1$ for $1\leq j\leq n$.
\end{rem}

\section{Parseval quasi-dual frames: infinite frames}\label{sec 4}

In this section we shall focus on infinite frames $\F:=\{f_i\}_{i\in \NN}$ for a separable and infinite dimensional Hilbert space $\cH$.
In this case $F$ will denote the synthesis operator of $\F$. In what follows we denote  
\[ \alpha(\F):= \inf \{ \|FX^*-I\|:\ XX^*=I_\cH \}\, , \] where $\|\cdot\|$ is the operator norm on $B(\cH)$. As explained before, $\alpha(\cF)$ is the optimal bound for the worst case reconstruction error in terms of a blind reconstruction algorithm based on Parseval frames. 
Moreover, in case that  $\alpha(\F)$ is attained, we are interested in the set of Parseval quasi-dual frames:
\[\mathfrak{X}(\F):=\{X: \, XX^*=I \ , \  \alpha(\F)=\|FX^*-I\| \} \, .\]

\subsection{Frames with infinite excess}

It is clear from Theorem \ref{hay dual} that $\alpha(\F)=0$ if and only if $A_\F\geq 1$ and $\overline{ R(S_\cF-I)}\leq \dim N(F)$. In particular, for frames with infinite excess, $\alpha(\F)=0$ if and only if $A_\F\geq 1$.
The following result provides upper and lower bounds for $\alpha(\F)$ if $\F$ does not have a Parseval dual, i.e. $\alpha(\F)>0$.

\begin{teo}\label{cota inferior}
Let $\F=\{f_i\}_{i\in \NN}$ be a frame in $\cH$ with optimal lower  frame bound $A_\F$ and synthesis operator $F$.
Assume that at least one of the following conditions holds:
\begin{enumerate}
\item  $A_\F<1$ ;
\item  $A_\F>1$ and $\dim N(F)<\dim \cH$ .
\end{enumerate}
 Then 
 \[ |1-A_\F\rai|\leq \alpha(\F)\leq \|S_\F\rai-I\|=\max\{1-A_\F\rai,\, B_\F\rai -1\}.\]
 \end{teo}
\begin{proof}
Let $F=S_\F\rai W$ be the (right) polar decomposition of $F$; then  $WW^*=I$ and $W^*W=P_{N(F)^\perp}$. Hence, 
\[
\alpha(\F)\leq \|FW^*-I\|=\|S_\F\rai-I\|.
\]
Suppose that $A_\F<1$ and that
there exists a coisometry $Y$ such that
\[\|FY^*-I\|\leq r <1-A_\F\rai. \]
Then, since $r<1$, $FY^*$ is invertible in $B(\cH)$ and its inverse $G=(FY^*)^{-1}$ satisfies  $\|G\|\leq \frac{1}{1-r}<A_\F\mrai$. 
Since $(GF)Y^*=I$, the  optimal lower frame bound $\hat{A}$ of $\{Gf_i\}_{i\in \NN}$ is $\hat{A}\geq 1$, by Theorem \ref{hay dual}.
Now, by Eq. \eqref{cotas optimas} we have that $\hat{A}\leq \|G\|^2 A_\F$ so then $1\leq \hat{A}\leq \|G\|^2 A_\F<1$,
which is a contradiction.

Now consider the case $A_\F>1$ and $\dim N(F)<\dim \cH$. As in the previous case, suppose that there exists  a coisometry $X$ such that
$\|FX^*-I\|<A_\F\rai-1$. 
\[
\|FX^*\|-1\leq \|FX^*-I\|<A_\F\rai -1 \ \ \Rightarrow \ \ \|FX^*\|<A_\F\rai \ .
\]
We claim that there exists a sequence $\{z_n\}_{n\in\NN}$ of unit norm vectors in $R(X^*)$ such that 
$\|P_{N(F)}z_n\|<\frac{1}{n}$. Indeed, let $P=P_{N(F)}$ and assume that there exists $\delta>0$ such that $\|Px\|\geq \delta \, \|x\|$ for every $x\in R(X^*)$. Then, $P|_{R(X^*)}:R(X^*)\rightarrow N(F)$ is an injective and closed range operator; hence $\dim N(F)\geq \dim R(X^*)=\dim(\cH)$ which is a contradiction,  so the claim is proved.

Consider now $x_n=Xz_n$, so that $\|x_n\|=1$ and $X^*x_n=z_n$ for $n\in\NN$, where $\{z_n\}_{n\in\NN}$ is as above. Then, by Equation \eqref{cotas optimas}  
 we get that 
\[A_\F\rai \left(1-\frac{1}{n}\right)\leq A_\F\rai\|(I-P)z_n\|\leq \|F(I-P)z_n\|=\|F z_n\|=\|FX^*x_n\|\leq \|FX^*\| <A_\F\rai \]
for $n\in\NN$, which implies that $A_\F\rai \leq \|FX^*\| <A_\F\rai$. Hence $\alpha(\F)\geq A_\F^{1/2}-1$ in this case.
\end{proof}

\begin{pro}\label{se alcanza}
Let $\F=\{f_i\}_{i\in \NN}$ be a frame in $\cH$ with synthesis operator $F$ and  optimal frame bounds $0<A_\F\leq B_\F $ such that $A_\F <1$.
Assume that some of the following conditions holds: 
\begin{enumerate}
\item $\dim \overline{R(S_\F-A_\F\, I)}\leq \dim N(F)$ ;
\item $A_\F\rai+B_\F\rai\leq 2$ .
\end{enumerate} 
Then $ \alpha(\F)= 1-A_\F\rai$ and in each case the infimum is attained.
\end{pro}
\begin{proof}
Notice that by assumption $A_\F<1$ and hence, by Theorem \ref{cota inferior}, we get that $\alpha(\F)\geq 1-A_\F\rai$.

Suppose  that $\dim \overline{R(S_\F-A_\F\, I)}\leq \dim N(F)$.
Let $T$ be the synthesis operator of $\{A_\F\mrai \,f_i\}_{i\in \NN}$, i.e. , $T=A_\F\mrai F$.
Notice that $TT^*=A_\F^{-1}S_\F\geq I$. 
Moreover, 
\[\dim \overline{ R(TT^*-I)}=\dim\overline{R(S_\F-A_\F\,I)}\leq \dim N(F)=\dim N(T).\]

Then, by Theorem \ref{hay dual} there exist a coisometry $Y$ such that
$TY^*=A_\F\mrai FY^*=I$ so then \[ \|FY^*-I\|=\|A_\F\rai \, I - I\|= 1-A_\F\rai \quad \Rightarrow \quad \alpha(\F)= 1-A_\F\rai \ . \] 
Now, suppose that $A_\F\rai+B_\F\rai\leq 2$, then $\|S_\F\rai-I\|=\max\{1-A_\F\rai,\, B_\F\rai -1\}= 1-A_\F\rai $ and the result follows by Theorem \ref{cota inferior}.
\end{proof}
The previous results show that the following conclusions hold in case the frame $\F$ satisfies that $\dim N(F)=\infty$.
\begin{teo}\label{resu prim part}
Let $\F$ be a frame for $\cH$ with   optimal lower frame bound $A_\F>0$ and synthesis operator $F$. If $\dim N(F)=\infty$, then:
\begin{enumerate}
\item $\alpha(F)=1-\min\{A_\F\rai, 1\}$.
\item The optimal lower bound $\alpha(F)$ is attained.
\item If $X\in\mathfrak{X}(\F)$ then $FX^*$ is an invertible operator in $B(\cH)$.
\item Moreover, we can choose $X\in\mathfrak{X}(\F)$ such that $FX^*=\min\{A_\F\rai, 1\}\, I$. 
\end{enumerate}
\end{teo} 
\begin{proof}
If $A_\F\geq 1$, all items are consequences of the existence of a Parseval dual for $\F$, since the conditions in item $\ref{condiciones}$ of Theorem \ref{hay dual} are clearly satisfied in this case.

If $A_\F<1$, then, by Proposition \ref{se alcanza} $\alpha(\F)=1-A_\F\rai$ and $\mathfrak{X}(\F)\neq \emptyset$. Moreover, since $\|FX^*-I\|=\alpha(\F)<1$ for every $\cX\in\mathfrak{X}(\F)$, we deduce that $FX^*$ is invertible in $B(\cH)$. Finally, in the proof of Proposition \ref{se alcanza} we see that there exists $X\in \mathfrak{X}(\F)$ such that  $FX^*=A_\F\rai I$.
\end{proof}

\begin{rem}
Let $\F$ be a frame for $\cH$ with  optimal lower frame bound $A_\F$ and synthesis operator $F$. Assume further that $\dim N(F)=\infty$. Then, according to Theorem \ref{resu prim part} there exists a Parseval frame $\cX=\{x_j\}_{j\in\NN}$ for $\cH$ such that $\|FX^*-I\|=\alpha(\F)$ and $FX^*= (1-\alpha(\F))\,I$. As a consequence of the previous facts we can obtain a perfect reconstruction algorithm derived from the pair $(\F,\cX)$, namely 
\[f=\beta(\F) \sum_{j\in\NN} \pint{f,f_j}\ x_j = \beta(\F) \sum_{j\in\NN} \pint{f,f_j}\ x_j \ , \ \ f\in \cH\ , \]
 where $\beta(\F)=(1-\alpha(\F))^{-1}=\max\{A_\cF^{-1/2},1\}$.  That is, in this case the tight frame $\beta(\F)\cdot \cX=\{\beta(\F) \, x_j\}_{j\in\NN}$ is a dual frame for $\F$ with frame bound $\beta(\F)$. The properties of $\alpha(\F)$ ensure that $\beta(\cF)\geq 1$ is the smallest constant for which there exists a tight frame $\cY$ for $\cH$ with frame bound $\beta$ and which is a dual frame for $\F$. 
\end{rem}
\subsection{Frames with finite excess.}\label{sec fra fin}
We are left to compute the value of $\alpha(\F)$ in case that $\cF$ is a frame for $\cH$ with $\dim (N(F))<\infty$. In order to do that, we shall relate this problem with the computation of the distance of a bounded operator to the group of unitary operators acting on a suitable closed subspace of $\ell_2=\ell_2(\NN)$. 

Indeed, let $\cM$ be an infinite dimensional closed subspace of $\ell_2$ and consider a coisometry $Y\in B(\ell_2,\cH)$ such that $Y^*Y=P_\cM$. Then
\begin{equation}\label{una ident con x}
\|FY^*-I\|=\|(FP_\cM - Y) Y^*\|=\|FP_\cM-Y\|\ .
\end{equation} Consider a fixed $Y$ as above and let $X$ be any other coisometry such that $X^*X=P_\cM$. Then 
\[\|FX^*-I\|=\|FP_\cM - X\|=\|Y^*(FP_\cM-X)\|=\|Y^*FP_\cM - Y^*X\|.\]
Notice that $\cM$ is an invariant subspace for both $Y^*F$ and $Y^*X$; moreover, if we consider the restrictions of these operators to $\cM$ we get
\[ \|Y^*FP_\cM - Y^*X\|= \|(Y^*FP_\cM - Y^*X )|_{\cM}\|=\|(Y^*F)|_{\cM} - (Y^*X)|_{\cM}\|\]
where $(Y^*F)|_{\cM}\in B(\cM)$ is fixed and $(Y^*X)|_{\cM}\in \cU(\cM)$ is a unitary operator acting on $\cM$. From the previous remarks we see that 
\begin{equation}\label{ec motivacion}
\inf \{\|FX^*-I\|\,:\; X^*X=P_{\cM},\; XX^*=I_\cH\}= d_{\cU(\cM)}((Y^*F)|_\cM).
 \end{equation}
 where $d_{\cU(\cM)}((Y^*F)|_{\cM})$ stands for the distance between the unitary group $\cU(\cM)$ of $B(\cM)$ and the operator $(Y^*F)|_{\cM}$.

These remarks show that there is a connection with our problem and the problem of computing the distance between an operator and the group of unitary operators in a Hilbert space. Since this is a rather technical matter, the proofs of the results related with this topic - which are based on \cite{Rogers} - will be presented in an Appendix.

In what follows we introduce some notions in terms of which we can give a complete description of the solution of our problem. These notions will play an important role in the proofs of the results developed in the Appendix.

%
%
%
%
%

Given an arbitrary operator $T\in B(\cH)$, we denote by $m(T)= \inf\{\|Tx\|, \quad \|x\|=1\}=\min \sigma(|T|)$. In addition,  if $\sigma_e(T)$ denotes the essential spectrum of $T$, let $m_e(T)=\min \sigma_e(|T|)$ and $\|T\|_e= \max \sigma_e(|T|)$. 
Moreover, for $n\in\NN$ let
\[u_n(T)=\sup \{\min \sigma(|T|_\cM)\; : \; \dim(\cM)=n\} \]
In this way we obtain the non-increasing sequence of non-negative real numbers $(u_n(T))_{n\in\NN}$. 

\begin{teo}\label{thm inf excess}
Let $\F=\{f_i\}_{i\in \NN}$ be a frame for $\cH$ with synthesis operator $F$ and optimal lower frame bound $A_\cF$. Suppose that $\dim(N(F))=n$ and let $C_\F=u_{n+1}(F)$. Then,
\[\alpha(\F)= \min\{\max\{C_\F-1, 1-A_\F\rai\},\, 1+m_e(F)\}\]
\end{teo}
\begin{proof}
See the Appendix. 
\end{proof}

Recall that the Fredholm index of $T$ is defined as $\ind(T)=\dim N(T)-\dim N(T^*)$ if at least one of these numbers is finite (for a detailed account on the theory of Fredholm index see \cite{Conbook}).

\begin{rem} With the notations of Theorem \ref{thm inf excess} above, it turns out that (see the proof of Theorem \ref{thm inf excess} in the Appendix) there always exists a infinite dimensional closed subspace $\cM$ such that, given a coisometry $Y$ with $R(Y^*)=\cM$, then 
\[\alpha(\cF)= d_{\cU(\cM)}((Y^*F)_\cM).\]
Consequently, the existence of Parseval quasi-dual frames  is subject to the existence of unitary approximants of $(Y^*F)_\cM$. Indeed, if $U\in \cU(\cM)$ is an approximant of $(Y^*F)_\cM$, then, let $\hat{U}\in U(\cH)$ be such that $\hat{U}=U\oplus U'$ where $U'\in \cU(\cM^\perp)$. In such case it is readily seen that $X=Y\hat{U}\in \mathfrak{X}(\cF)$.

By \cite[Theorem 1.4]{Rogers}, if  $\ind(Y^*FP)=0$
(e.g., if $\alpha(\F)<1$), there are  unitary approximants for $(Y^*F)_\cM$. In particular, $\mathfrak{X}(\F)$ is not empty.

Denote by  $E(\cdot)$ the spectral measure of $|F|$. If $\ind(Y^*FP)\neq 0$ (which implies $\alpha(\F)=m_e(F)+1$) and
  $E([0, m_e(F)))$ is an orthogonal projection of infinite rank or if $m_e(F)$ is a cluster point of eigenvalues of $|F|$, then  the existence of unitary approximants and hence Parseval quasi-dual frames is ensured by   \cite[Theorem 1.4, (iii)]{Rogers}. 
\end{rem}

\begin{rem}\label{final rem}
Let $\F$ be a frame for $\cH$ that admits Parseval quasi-dual frames, i.e. such that $\mathfrak{X}(\F)\neq\emptyset$.
In this case, it seems interesting to search for optimal Parseval quasi-dual frames $X\in\mathfrak{X}(\F)$, i.e. such that 
\[\|F-X\|=\inf\{ \|F-Y\|: \, \cY\in\mathfrak{X}(\F)\}\ . \]
In case $\alpha(\F)=0$ then, for every $X\in\mathfrak{X}(\F)$ we have that   
\[ (F-X)(F-X^*)=S_\F-FX^*-XF^*+XX^*=S_\F-I \ .\] 
Thus, $\|F-X\|=\|S_\F-I\|^{1/2}$ and every Parseval quasi-dual frame $X\in\mathfrak{X}(\F)$ is optimal.

 If  $\alpha(\F)>0$ 
 then it is no longer true that $\|F-Y\|$ is constant for $Y\in \mathfrak{X}(\F)$. 
Indeed, 
 suppose that $F$ is such that  $\dim \overline{R(S_\F - A_\F\, I)}\leq \dim N(F)$ and $A_\F\rai +B_\F\rai\leq 2$. Then, by the proof of Proposition \ref{se alcanza}, there is a coisometry $Y$ such that $FY^*=A_\F\rai I$ and $\|FY^*-I\|=1-A_\F\rai=\alpha(\F)$. Then,
\[(F-Y)(F-Y)^*=S_\F+(1-2 \, A_\F\rai)I,\]
which implies that $\|F-Y\|=(B_\F+1-2A_\F\rai)\rai$. On the other hand, the proof of Proposition \ref{se alcanza} also shows that the coisometry $W$ of the polar decomposition of $F$ also satisfies $\|FW^*-I\|=\alpha(\F)$. But it is easy to see that $\|F-W\|=1-A_\F\rai$ which is equal to $B_\F+1-2A_\F\rai$ only if $B_\F=A_\F$.
\end{rem}

\section{Appendix}\label{App}

In this section we revise some facts related with unitary approximations of bounded operators in Hilbert spaces and we present the proof of Theorem \ref{thm inf excess}. Throughout this section $\cF$ denotes a frame with finite excess for an infinite dimensional Hilbert space $\cH$.  So that, if  $F\in B(\ell_2,\cH)$ denotes the frame operator of $\cF$,  $\,\dim (N(F))<\infty$.

Recall that given $T\in B(\cH)$, then $m(T)= \inf\{\|Tx\|, \quad \|x\|=1\}=\min \sigma(|T|)$, $\sigma_e(T)$ denotes the essential spectrum of $T$, $m_e(T)=\min \sigma_e(|T|)$ and $\|T\|_e= \max \sigma_e(|T|)$.

We begin by recalling one of the main results from \cite{Rogers}.
\begin{teo}[Theorem 1.3 in \cite{Rogers}]\label{Rogers} Let $T\in B(\cH)$, then,
\[d_{\cU(\cH)}(T):=\inf_{U\in \cU(\cH)} \|T-U\|=\begin{cases} &\max\{ \|T\|-1, 1-m(T)\} \; \mbox{ if } \ind(T)=0\\
                                            &\max \{\|T\|-1, 1+m_e(T)\} \;\mbox{ if } \ind(T)<0.
                                            \end{cases}   
\]                                            
(the case $\ind(T>0)$ follows  using $T^*$).
\end{teo}  

Let $\cM$ be an infinite dimensional closed subspace of $\ell_2$ and consider a coisometry $Y\in B(\ell_2,\cH)$ such that $Y^*Y=P_\cM$.
Recall that our interest in the unitary approximation problem is motivated by Eq. \eqref{ec motivacion}. Hence, in order to apply Theorem \ref{Rogers} to the operator $(Y^*F)|_\cM\in B(\cM)$, we need to relate the Fredholm index of this operator to that of $F$ (which is finite, since we are assuming that $F$ is surjective and that $N(F)$ has finite dimension). In order to describe such a relation we introduce the following notation: given $T\in B(\cH)$ and $\cM\subseteq \cH$ a closed subspace then $T_\cM=P_\cM T|_\cM\in B(\cM)$ denotes the compression of $T$ to $\cM$.

\begin{lem}\label{indice compr}Let $T\in B(\cH)$. 
\begin{enumerate}
\item Let $\cM\subseteq \cH$ be a closed subspace with $\dim \cM^\perp<\infty$. If $\ind(T)\in \mathbb Z$ then $\ind(T_\cM)=\ind(T)$.
\item If $T$ is a closed range operator with $\ind(T)=-\infty$, then $\ind(T_{R(T)})=-\infty$.
\end{enumerate}
\end{lem}
\begin{proof}
1. Let $P$ be the orthogonal projection onto $\cM$ and let $m=\dim N(P)=\dim \cM^\perp<\infty$. 
Notice that $\dim N(T_\cM) = \dim N(PTP)-m$ and $\dim N(T^*_\cM) =\dim N(PT^*P) -m$ so that $\ind(PTP)=\ind(T_\cM)$. Moreover,  $T=PTP+PT(I-P)+(I-P)T=PTP+K$, where $K$ is a finite rank operator. Then, $\ind(T)=\ind(PTP)=\ind(T_\cM)$.

2. Suppose now that $T$ is a semi-Fredholm operator  with $\ind(T)=-\infty$. Let $\cM=R(T)$, $\cN=N(T)$ and notice that by hypothesis $\dim \cM^\perp=\dim N(T^*)=\infty$, $\dim \cN<\infty$. This last fact shows that $\dim \cM=\infty$.

In this case we have that $N(T_\cM)=\cN\cap \cM$, and $\dim N(T_\cM \, ^*)=\dim \cN^\perp\cap \cM^\perp$; indeed, the first identity easily follows from the definition of $T_\cM$. For the second identity, notice that $T_\cM\,^*=P_\cM \,T^*|_\cM$ and $N(T_\cM^*)\subseteq \cM=N(T^*)^\perp$; hence $T^*|_{N(T_\cM\,^*)}:N(T_\cM\,^*)\rightarrow R(T^*)\cap \cM^\perp= \cN^\perp\cap\cM^\perp$ is a linear isomorphism i.e. a linear transformation with bounded inverse.

 Let $X,\,Y\in B(\cH)$ be coisometries with initial space $\cM$ and $\cN^\perp$ respectively. Hence, $\ind(X)=\dim \cM^\perp=\infty$ and $\ind(Y)=\dim \cN$.  Then, by the additivity property of the index for (left) semi-Fredholm operators, $\ind(YX^*)=\ind(Y)+\ind(X^*)=-\infty$. On the other hand, arguing as before it is easy to see that 
 $\dim N(XY^*) =\dim \cM\cap \cN$ and $\dim N(YX^*) =\dim \cM^\perp\cap \cN^\perp$. Hence, $\ind(T_\cM)=\dim \cM\cap \cN -  \dim \cM^\perp\cap \cN^\perp=-\infty$.
\end{proof}

Given $T\in B(\cH)$, a useful way to compute $m_e(T)$ and $\|T\|_e$ is by using  the maps 
\[U_k(T)=\sup_{E\in \cP(\cH), \; \tr E\leq k} \tr (|T|E) \quad \text{ and }\quad L_k(T)=\inf_{ E\in \cP(\cH), \; \tr E\leq k} \tr(|T|E)\, , \] where $\cP(\cH)$ denotes the set of orthogonal projections in $B(\cH)$ and $\tr(\cdot)$ denotes the usual (semifinite) trace in $B(\cH)$.
Indeed, by \cite[Prop. 3.5]{AMRS} we have that 
 \begin{equation}\label{eq me ne}
 m_e(T)=\lim_{k\rightarrow \infty} \, \frac{L_k(T)}{k} \quad \text{ and } \quad \|T\|_e=\lim_{k\rightarrow \infty} \, \frac{U_k(T)}{k}.
 \end{equation}

On the other hand, for $n\in\NN$ we let
\[u_n(T)=\sup \{\min \sigma(|T|_\cM)\; : \; \dim(\cM)=n\} \quad \text{and} \quad l_n(T)=\inf \{\max \sigma(|T|_\cM)\; : \; \dim(\cM)=n\}.\]
In this way we obtain the non-increasing sequence $(u_n(T))_{n\in\NN}$ and the non-decreasing sequence $(l_n(T))_{n\in\NN}$. 
Denote by  $E(\cdot)$ the spectral measure of $|T|$. 
Then, it is easy to see that $u_n=\|T\|_e$ if the range of the projection $E((\|T\|_e, \|T\|])$ is a subspace of dimension $k<n$. Otherwise $u_n=\la_n$ if $\la=(\la_i)_{i=1}^m$ are the eigenvalues (counting multiplicity) arranged in a decreasing order of the finite rank operator $E(I)|T|E(I)$, where $I\subset (\|T\|_e, \|T\|]$ is any interval such that $n<\rk(E(I))=m<\infty$. There is an obvious analogue for $l_n$ using the eigenvalues of $|T|$ strictly smaller than $m_e(T)$.  

\begin{pro}\label{espectro compr} Let $T$ be a positive operator with $\dim(N(T))=n$ and let $\cM\subseteq \cH$ be a closed subspace with $\dim \cM=\infty$. Let $u_{n+1}=u_{n+1}(T)$ and $l_{n+1}=l_{n+1}(T)$.
 Then,
\begin{enumerate}
\item $m_e(T)\leq m_e(T_\cM).$ 
\item If we assume further that $\dim \cM^\perp =n$ then $m(T_\cM)\leq l_{n+1}\leq u_{n+1}\leq \|T_\cM\|.$
\end{enumerate}
Moreover,  for every $\eps>0$ there exist infinite dimensional closed subspaces  $\cN_\eps$ and $\cM$ such that  $\dim(\cM^\perp)=n$ and 
\begin{enumerate}
\item[3.] $m(T_{\cM})= l_{n+1}$ and $\|T_{\cM}\|= u_{n+1}$.
\item[4.] $m_e(T_{\cN_\eps})=m_e(T)$ and $\|T_{\cN_\eps}\|\leq m_e(T)+\eps$.
\end{enumerate}
\end{pro}

\begin{proof}
Let $P$ be the orthogonal projection onto $\cM$. Then, by Eq. \eqref{eq me ne} we get that
\begin{align*} 
m_e(T_\cM)&=\lim_{k\rightarrow \infty} \frac{1}{k} \inf\{\tr(T_\cM E)\, : \,E\in \cP(\cM), \, \tr(E)\leq k\}\\
&=\lim_{k\rightarrow \infty} \frac{1}{k} \inf\{\tr(TE)\, : \,E\in \cP(\cH), \, \tr(E)\leq k, \; R(E)\subset \cM\}\\
&\geq \lim_{k\rightarrow \infty} \frac{1}{k} \inf\{\tr(TE)\, : \, E\in \cP(\cH) \, \tr(E)\leq k\}=m_e(T)\, .
\end{align*}

Assume further that $\dim \cM^\perp=n$. Given $\eps>0$, let $\cS_\eps$ be a $n+1$-dimensional subspace of $\cH$ such that $\min \sigma(T_{\cS_\eps})>u_{n+1}-\eps$. We claim that $\cS_\eps\cap \cM\neq \{0\}$: indeed, if $\cS_\eps\cap \cM= \{0\}$ then $P_{\cM^\perp}|_{\cS_\eps}:\cS_\eps\rightarrow \cM^\perp$ is an injection, which contradicts the fact that $\dim \cS_\eps>\dim \cM^\perp$. Therefore, if $x\in \cS_\eps\cap\cM$ with $\|x\|=1$ then 
\[\pint{ T_\cM  x\, , \, x}=\pint{T x\, , \, x}=\pint{T_{\cS_\eps} x\, , \, x} \geq u_{n+1}-\eps \, .\]
Since $\eps$ was arbitrary, we see that  $\|T_\cM\|\geq u_{n+1}$. The proof for the lower bound is similar.

In order to finish the proof, we exhibit the subspaces $\cM_\eps$ and $\cN_\eps$ as above.
If $\|T\|=\|T\|_e$ we just take $\cM_\eps=N(T)^\perp$ and we are done. In case that $\|T\|>\|T\|_e$, define $r=\dim R(E\left((\|T\|_e, \|T\|]\right)$ and let $k=\min\{n, r\}$. Notice that in this case, $u_1\geq \cdots\geq u_k$ are eigenvalues of $T$. Denote by $\cS=N(T)\oplus \cE$, where $\cE$ is the $k$-dimensional subspace generated by eigenvectors associated to $u_1, \ldots, u_k$. Notice that $\|T_{\cS^\perp}\|=u_{k+1}$ and that the $n+k$ eigenvalues of $T_\cS$ (counting multiplicities and arranged in non-increasing order) are $u_1,\ldots,u_k, 0,\ldots,0$. Therefore, by Theorem \ref{FanPallTeo} there exists a $k$-dimensional subspace of $\cS$, denoted by $\cT$ such that $P_\cT TP_\cT =u_{k+1}P_\cT$. Thus, if we define 
$\cM=\cS^\perp\oplus \cT$ we obtain a subspace with $\dim \cM^\perp=n$ and such that $\|T_\cM\|=u_{k+1}$. Therefore, if $n<r$ (and hence $k=n$) we see that $\|T_\cM\|=u_{k+1}=u_{n+1}$; otherwise $n\geq r$ (so that $k=r$) and hence $\|T_\cM\|=u_{r+1}=\|T\|_e=u_{n+1}$.

Finally, if $Q=E((m_e(T)-\eps, m_e(T)+\eps))$ then $Q$ is a orthogonal projection with infinite dimensional range $\cN_\eps=R(Q)$ and $m_e(T_{\cN_\eps})=m_e(T)$, $\|T_{\cN_\eps}\|\leq m_e(T)+\eps$.
\end{proof}

Next we present the proof of the main result of Section \ref{sec fra fin}. Recall that the Fredholm index of $T$ is defined as $\ind(T)=\dim N(T)-\dim N(T^*)$ if at least one of these numbers is finite. 

\begin{proof}[of Theorem \ref{thm inf excess}]
Let $Y\in B(\ell_2,\cH)$ be a fixed coisometry with initial space $\cM\subset \ell_2$ i.e. $YY^*=I_\cH$ and $Y^*Y=P_\cM$, where $P_\cM$ denotes the orthogonal projection onto $\cM$. Then, as explained at the beginning of Section \ref{sec fra fin} (see Eq. \eqref{ec motivacion})
\[\inf \{\|FX^*-I\|\,:\; X^*X=P_{\cM},\; XX^*=I_\cH\}= d_{\cU(\cM)}((Y^*F)_\cM).\]
Therefore, we have that 
\begin{equation}\label{calc alp}
 \alpha(\F)=\inf\{d_{\cU(\cM)}((Y^*F)_\cM): YY^*=I_\cH\, , \ Y^*Y=P_\cM\}\ .
\end{equation}
Notice that $N((Y^*F)_\cM)=N(F)\cap \cM$ and that $\dim N((Y^*F)_\cM^*)=\dim R(F^*)\cap\cM^\perp$. In particular, $\dim N((Y^*F)_\cM)\leq n$ and therefore $\ind((Y^*F)_\cM)\leq n$ and it is well defined.

Now, we claim that $\ind((Y^*F)_\cM)=0$ if and only if $\dim(\cM^\perp)=n$. 
Indeed, if we assume that $\dim(\cM^\perp)=\infty$, then
$R(Y^*F)=\cM$ and $\ind(Y^*F)=-\infty$; hence, by Lemma \ref{indice compr} we see that $\ind((Y^*F)_\cM)=-\infty$ in this case. On the other hand, if 
$\dim\cM^\perp=m$ 
 then, by  Lemma \ref{indice compr} and the additivity of the Fredholm index for (left) semi-Fredholm operators, $\ind((Y^*F)_\cM)=\ind(Y^*F)=\ind(Y^*)+\ind(F)=n-m$. 

Hence, if we take $\cM\subseteq\ell_2$ such that $\dim (\cM^\perp)=n$ then the previous facts together with Theorem  \ref{Rogers} imply that 
\[ d_{\cU(\cM)}((Y^*F)_\cM)=\max\{\| \, |F|_\cM\| -1 \, , \, 1-m(\,|F|_\cM)\}\]
since $|(Y^*F)_\cM|^2=(|F|^2)_\cM$ so that $\|\, |(Y^*F)_\cM|\, \|= \| |F|_\cM\|$ and $m(|(Y^*F)_\cM|)=m(|F|_\cM)$. Moreover, using the fact that $l_{n+1}(|F|)=A_\F$ and $C_\F=u_{n+1}(|F|)$, then
items 1. and 3. in Proposition \ref{espectro compr} show that 
\begin{equation}\label{eq teo 1}\inf \{d_{\cU(\cM)}((Y^*F)_\cM) :\, YY^*=I_\cH \, , \, Y^*Y=P_\cM \, , \, \ind((Y^*F)_\cM)=0
\}=\max\{1-A_\F\rai, C_\F-1\} .
\end{equation}
On the other hand, if $\dim(\cM^\perp)\neq n$, then $\ind((Y^*F)_\cM)\neq 0$. 
Thus, by Theorem \ref{Rogers} and Proposition \ref{espectro compr} we conclude that 
\begin{equation}\label{eq teo 2} \inf \{d_{\cU(\cM)}((Y^*F)_\cM)\, :\, YY^*=I_\cH\, , \  Y^*Y=\cM 
\, , \ \ind((Y^*F)_\cM)<0 \}
=1+m_e(F)\ .\end{equation}
Finally, if we assume that $\ind((Y^*F)_\cM)>0$ then, as shown above, $\dim \cM^\perp=m<n<\infty$. Thus 
$|(Y^*F)_\cM^*|^2= Y^*(F\,P_\cM F^*)Y|_\cM$ and hence, since $Y$ is a coisometry with initial space $\cM$, we see that $m_e((Y^*F)_\cM^*)=m_e(F\,P_\cM\, F^*)^{1/2}$. Now, notice that $\dim N(F)<\infty$ implies that $m_e(FF^*)=m_e(F^*F)$
(e.g. \cite[Proposition 4.5]{AMRS} shows that $(FF^*\oplus \,0_n)$ and $F^*F$ are unitarily equivalent, where $0_n$ is the zero operator acting on an $n$-dimensional Hilbert space). Then,
 \[
m_e((Y^*F)_\cM^*)=\,m_e(F\,P_\cM\,F^*)\rai=\, m_e(FF^*)\rai =\, m_e(F^*F)\rai=\, m_e(F)\, , 
 \]
  since $FF^*=F\,P_\cM\,F^*+F\,P_{\cM^\perp}\,F^*$ and $F\,P_{\cM^\perp}\,F^*$ is a finite rank operator. Thus, by Theorem \ref{Rogers} we get that
\begin{equation}\label{eq teo 3} \inf \{d_{\cU(\cM)}((Y^*F)_\cM)\, :\, YY^*=I_\cH\, , \  Y^*Y=\cM 
\, , \ \ind((Y^*F)_\cM)>0 \}
\geq 1+m_e(F)\ .\end{equation}
The result follows by combining Eqs. \eqref{calc alp}, \eqref{eq teo 1}, \eqref{eq teo 2} and \eqref{eq teo 3}.
\end{proof}

\noindent {\bf Acknowledgements:} We thank Professor Jorge Antezana for useful suggestions regarding the material in this note. This work is partially supported by UBACYT I023,  CONICET (PICT  808/08 and PIP 0435) and UNLP 11X585.

\medskip

\noindent G. Corach \\
e-mail: {gcorach@fi.uba.ar}

\medskip

\noindent P. Massey\\
e-mail: {massey@mate.unlp.edu.ar}

\medskip

\noindent M. Ruiz\\
e-mail: {mruiz@mate.unlp.edu.ar}




\begin{thebibliography}{99}
\bibitem{ACRS} Antezana, J., Corach, G., Ruiz, M.,  Stojanoff, D.: Oblique projections and frames, Proc.
Amer. Math. Soc. 134, 1031 - 1037 (2006).

\bibitem{AMRS} Antezana, J., Massey, P., Ruiz, M., Stojanoff, D.: The Schur-Horn Theorem for
operators and frames with prescribed norms and frame operator, Illinois J. Math.
51, 537-560 (2007).

\bibitem{Bhatia} Bhatia, R.: Matrix analysis. Graduate Texts in Mathematics, 169. Springer-Verlag, New York, (1997).

\bibitem{BLem} Bownik, M., Lemvig,J.: The canonical and alternate duals of a wavelet frame. Appl. Comput. Harmon. Anal. 23, no. 2, 263-272 (2007).

\bibitem{Crisbook}Christensen, O.: An introduction to frames and Riesz bases. Applied and Numerical Harmonic Analysis. Birkhäuser Boston, Inc., Boston, MA, (2003).

\bibitem{CE} Christensen O., Eldar, Y.C.:  Oblique dual frames and shift-invariant spaces, Appl. Comput. Harmon. Anal. 17,  48-68 (2004).

\bibitem{Conbook} Conway, J.B.: A course in functional analysis. Second edition. Graduate Texts in Mathematics, 96. Springer-Verlag, New York, (1990).

\bibitem{ElPa} Eld\'en, L., Park, H. :A Procrustes problem on the Stiefel manifold.Numer. Math. 82, no. 4, 599-619 (1999).

\bibitem{FanHof} Fan, K., Hoffman A.J.: Some metric inequalities in the space of matrices, Proc. Amer. Math. Soc. 6, 111-116 (1955).

\bibitem{FanPall} Fan, K., Pall, G.: Imbedding conditions for Hermitian and normal matrices. Canad. J. Math. 9, 298-304 (1957).

\bibitem{GoDi}Gower, J. C., Dijksterhuis, G. B.: Procrustes problems. Oxford Statistical Science Series, 30. Oxford University Press, Oxford, (2004).

\bibitem{Grochbook} Gr\"ochenig, K.: Foundations of time-frequency analysis. Applied and Numerical Harmonic Analysis. Birkhäuser Boston, Inc., Boston, MA, (2001).

\bibitem{Han} Han, D.: Frame representations and Parseval duals with applications to Gabor frames, Trans. Amer. Math. Soc.  360, 3307-3326 (2008).

\bibitem{HKLW} Han, D., Kornelson, K., Larson, D., Weber, E.: Frames for undergraduates, Student
Mathematical Library, 40. American Mathematical Society, Providence, RI, (2007).

\bibitem{HanLar} Han, D. , Larson, D. : Frames, bases and group representations, Mem. Amer. Math. Soc.
147, no. 697(2000).

\bibitem{NHi} Higham, N.J.:  Matrix nearness problems and applications. Applications of matrix theory (Bradford, 1988), 1-27, Inst. Math. Appl. Conf. Ser. New Ser., 22, Oxford Univ. Press, New York, (1989).

\bibitem{Kni} Kintzel, U.: Procrustes problems in finite dimensional indefinite scalar product spaces. Linear Algebra Appl. 402, 1-28 (2005).

\bibitem{LHan}Leng, J.,  Han, D.: Optimal dual frames for erasures II. Linear Algebra Appl. 435, no. 6, 1464-1472 (2011). 

\bibitem{LHantin} Leng, J., Han, D.,  Huang, T.: Optimal dual frames for communication coding with probabilistic erasures. IEEE Trans. Signal Process. 59, no. 11, 5380-5389 (2011).

\bibitem{LH} Lopez, J. ,  Han, D.: Optimal dual frames for erasures. Linear Algebra Appl. 432, no. 1, 471-482 (2010).

\bibitem{MRSrob} Massey, P., Ruiz,  M., Stojanoff, D.: Robust dual reconstruction systems and fusion frames. Acta Appl. Math. 119, 167-183 (2012).

\bibitem{MRS} Massey, P., Ruiz,  M., Stojanoff, D.: Optimal dual frames and frame completions for majorization,  Appl. Comput. Harmon. Anal. 34, no. 2, 201-223 (2013).

\bibitem{Mat} Mathias, R.: Perturbation bounds for the polar decomposition. SIAM J. Matrix Anal. Appl. 14, no. 2, 588-597 (1993).

\bibitem{PeHuZh}Peng, J., Hu, X. Y., Zhang, L.: The $(M,N)$-symmetric Procrustes problem. Appl. Math. Comput. 198 , no. 1, 24-34 (2008).

\bibitem{Rogers}Rogers, D.: Approximation by unitary and essentially unitary operators, Acta Sci. Math. 39 , 141-151 (1977).

\bibitem{WES} Werther, T.,  Eldar, Y.C.,  Subbanna,  N.: Dual Gabor frames: theory and computational aspects. IEEE Trans. Signal Process. 53, no. 11, 4147-4158 (2005).

\bibitem{Wat1} Watson, G. A.: The solution of orthogonal Procrustes problems for a family of orthogonally invariant norms. Adv. Comput. Math. 2 , no. 4, 393-405 (1994).

\bibitem{Wat2} Watson, G. A.: Solving generalizations of orthogonal Procrustes problems.Contributions in numerical mathematics, 413-426, World Sci. Ser. Appl. Anal., 2, World Sci. Publ., River Edge, NJ, (1993).

\end{thebibliography}
\end{document}